\documentclass[letterpaper, 10 pt, conference]{ieeeconf}  

\IEEEoverridecommandlockouts                              
\usepackage{bbm}
\usepackage{amsfonts}
\usepackage{mathrsfs}
\usepackage{epsfig}
\usepackage{epsf}
\usepackage{epstopdf}
\usepackage{algorithm}
\usepackage{algpseudocode}
\usepackage{graphics} 
\usepackage{epsfig} 
\usepackage{times} 
\usepackage{amsmath} 
\usepackage{amssymb}  
\usepackage{subfigure}
\usepackage{graphicx}
\usepackage{flushend}
\usepackage{array,color,arydshln}
\usepackage{booktabs}
\usepackage{xcolor}
\usepackage{amsmath}
\usepackage{amssymb}
\usepackage[normalem]{ulem}
\usepackage{color}
\newcommand\redsout{\bgroup\markoverwith{\textcolor{red}{\rule[0.55ex]{2pt}{0.4pt}}}\ULon} 
\usepackage{xcolor} 

\newtheorem{theorem}{Theorem}

\newtheorem{lemma}{Lemma}

\newtheorem{definition}{Definition}

\newcommand{\rank}{{\rm rank\;}}

\usepackage{multicol}
\usepackage{multirow}
\usepackage{booktabs}

\usepackage{verbatim} 

\newcommand{\revision}[1]{\textcolor{black}{#1}}

 \newcommand{\todo}[1]{\textcolor{magenta}{#1}} 
 \newcommand{\keypoint}[1]{\textcolor{black}{#1}} 

\title{\LARGE \bf
Formation control for multiple agents with local measurements: continuous-time and sampled-data-based cases
}

\author{Chen~Wang, Shuai~Li, Weiguo~Xia, Jinan~Sun and~Guangming~Xie 
\thanks{This work was supported in part by grants from the National Natural Science Foundation of China (NSFC, \revision{No. 61973007, 61973051, 91648120, 61603071, 61633002}), the Fundamental Research Funds for the Central Universities under Grant DUT19ZD103, and the Youth Star of Dalian Science and Technology (2018RQ51).}
\thanks{C. Wang and J. Sun are with the National Engineering Research Center for Software Engineering, Peking University, Beijing 100871, China.	
       {\tt\small \{wangchen, sjn\}@pku.edu.cn}}
\thanks{W. Xia is with the Key Laboratory of Intelligent Control and Optimization for Industrial Equipment
	of Ministry of Education, and is also with the School of Control Science and Engineering,
	Dalian University of Technology, Dalian 116024, China.
       {\tt\small wgxiaseu@dlut.edu.cn}}
\thanks{C. Wang, S. Li and G. Xie are with the State Key Laboratory of Turbulence and Complex Systems,
	Intelligent Biomimetic Design Lab, College of Engineering, Peking University, Beijing 100871, China.
       {\tt\small \{wangchen, shuaier, xiegming\}@pku.edu.cn}}
       }

\begin{document}
%
%

\maketitle
\thispagestyle{empty}
\pagestyle{empty}
%

%
%

\begin{abstract}
    We study the formation control problem for a group of mobile agents in a plane, in which each agent is modeled as a kinematic point and can only use the local measurements in its \revision{local} frame. The agents are required to maintain a geometric pattern while keeping a desired distance to a static/moving target. The prescribed formation is a general one which can be any geometric pattern, and the neighboring relationship of the $N$-agent system only has the requirement of containing a directed spanning tree. To solve the formation control problem, a distributed controller is proposed based on the idea of decoupled design. One merit of the controller is that it only uses each agent's local measurements in its \revision{local} frame, so that a practical issue that the lack of a global coordinate frame or a common reference direction for real multi-robot systems is successfully solved. Considering another practical issue of real robotic applications that sampled data is desirable instead of continuous-time signals, the sampled-data based controller is developed. Theoretical analysis of the convergence to the desired formation is provided for the multi-agent system under both the continuous-time controller \revision{with a static/moving target} and the sampled-data based one \revision{with a static target}. Numerical simulations are given to show the effectiveness and performance of the controllers.
\end{abstract}
%

\par~

%
\section{Introduction}\label{se:introduction}
%



Cooperation of a group of robots has wide practical potential
in various applications \cite{DuMa12,ChHoIs11,DiYaLi12}. 
In such cooperative tasks,
the robots can benefit from moving in formation with certain desired geometric shapes \cite{OhPaAh15,BuCoMa09}.
Thus formation control of multi-robot systems has captured increasing attention \cite{OhPaAh15}.
However, the restrictions in application implementations cause the theoretical challenges of
controlling multiple robots.

One key theoretical challenge of such formation control problems for multi-agent systems arises from the fact that
the centralized coordination may not be allowed,
so that the robots can use only local information to implement their distributed control strategies.
To overcome this challenge, a considerable number of studies have focused on consensus based formation control
where the formation control problem is converted
 to a state consensus problem.
Specifically, the dynamics of the agents are modeled as single-integrators \cite{SoLiFe16,WaXi17a}, double-integrators \cite{ShLiTe15}, and unicycles \cite{MaBrFr04, BrSeCa14, YuLi16a};
some constrained conditions are considered including input saturation \cite{SoLiFe16}, agents' locomotion constraints \cite{WaXiCa12b}, finite-time control \cite{WaXiCa12a}, and limited communication \cite{GaCaYuAnCa13}.
Among these, lots of research efforts have been devoted to the target circular formation problem.
\cite{ChSuLiWa17} has proposed swarm control laws to realize some formation configurations of large-scale swarms using the nonlinear bifurcation dynamics.
However, 
no theoretical analysis was provided.
In \cite{WaXi17}, controlling a group of agents to form a circle around a prescribed target was studied,
where collision avoidance among agents has been guaranteed. 
However, the desired formation is limited to circles and only \revision{the} \revision{continuous}-time case is considered.

Another key theoretical challenge is 
that in lots of situations the robots can only use their local measurements without knowing the global coordinate frame or the common reference direction.
\cite{WaXiSuFaXi18} considered the formation problem for a group of mobile agents to maintain a prescribed distribution pattern.
The proposed controller can be implemented in each agent's \revision{local} frame so that the challenge of lack of a global coordinate frame or a common reference direction has been overcome.
However, the neighboring relationship among agents is restricted to a special one described by a ring topology;
and they haven't considered the case that sampled-data is desirable instead of continuous-time signals.

The goal of this paper is to design a distributed controller that can guide a group of mobile agents in a plane to form any given formation. 
The general control objective of the problem comprises two specific sub-objectives.
One is to form a desired geometric pattern where each pair of agents converges to a desired distance.
The other is to keep the formed geometric pattern rotating around a static/moving target when keeping a desired distance to the target.
It is worth to emphasize that the geometric patterns here allow that the distances between neighbors are distinguished and the distances from the agents to the target are different.
Thus the desired formation can be any geometric pattern.
We consider a system consisting of multiple agents modeled as single integrators.
The agents can only have local measurements in their own \revision{local} frame without knowing the information or a global coordinate frame of a common reference direction.

To realize the formation, a decoupled design is delivered in this paper.
We propose to use a controller for each agent comprised of two parts to deal with the two sub-objectives of the formation control problem, respectively.
Specifically, the designed controller is presented in each agent's \revision{local} frame,
since only the local measurements are accessible.
%

The main contributions of this paper is threefold.
First, we investigate the formation control problem only using each agent's local measurements in its \revision{local} frame, so that a practical issue that the lack of a global coordinate frame or a common reference direction for real multi-robot systems is successfully solved.
Second, we take into account another practical issue that real robotic applications requires sampled data instead of continuous-time signals, so that the sampled-data based controller is developed and analyzed.
Third, both the continuous-time controller and the sampled-data one have a nice property that
some parameters in the designed controller have explicit physical meanings,
so that these parameters can be selected more reasonable and easily according to the request of the robots' motion characteristics when applied to real robot systems in the future.

The rest of the paper is organized as follows.
In Section \ref{se:problem}, we formulate the formation control problem and give some useful preliminary results.
Then we design a distributed controller  using only the local measurements of the agents and provide rigorous analysis on its performances in Section \ref{se:continuous}.
In Section \ref{se:sampled}, a sampled-data based control law is investigated.
Simulation results are given in Section \ref{se:simulations}.
Finally, Section \ref{se:conclusions} concludes this paper.
%



\emph{\keypoint{Notation}}:
$\mathbb{R}$ denotes the set of real numbers.
$|\cdot|$ describes the absolute value of a scalar or the number of elements in a set.
For a matrix $A$,
 $A^T$, $\|A\|$ and $\rank(A)$ denote its transpose, Euclidean norm, and rank, respectively.
%

\par~

%
\section{Problem formulation and preliminaries}\label{se:problem}
%

In this section, we first formulate the problem of formation control for a group of mobile agents  using only their local measurements, and then give some useful preliminary results.

\begin{figure}[thpb]
\begin{center}
           \subfigure[\revision{Agent $i$'s local} frame]{\label{fig:Polar_relative}
           	\includegraphics[width=0.5\linewidth]{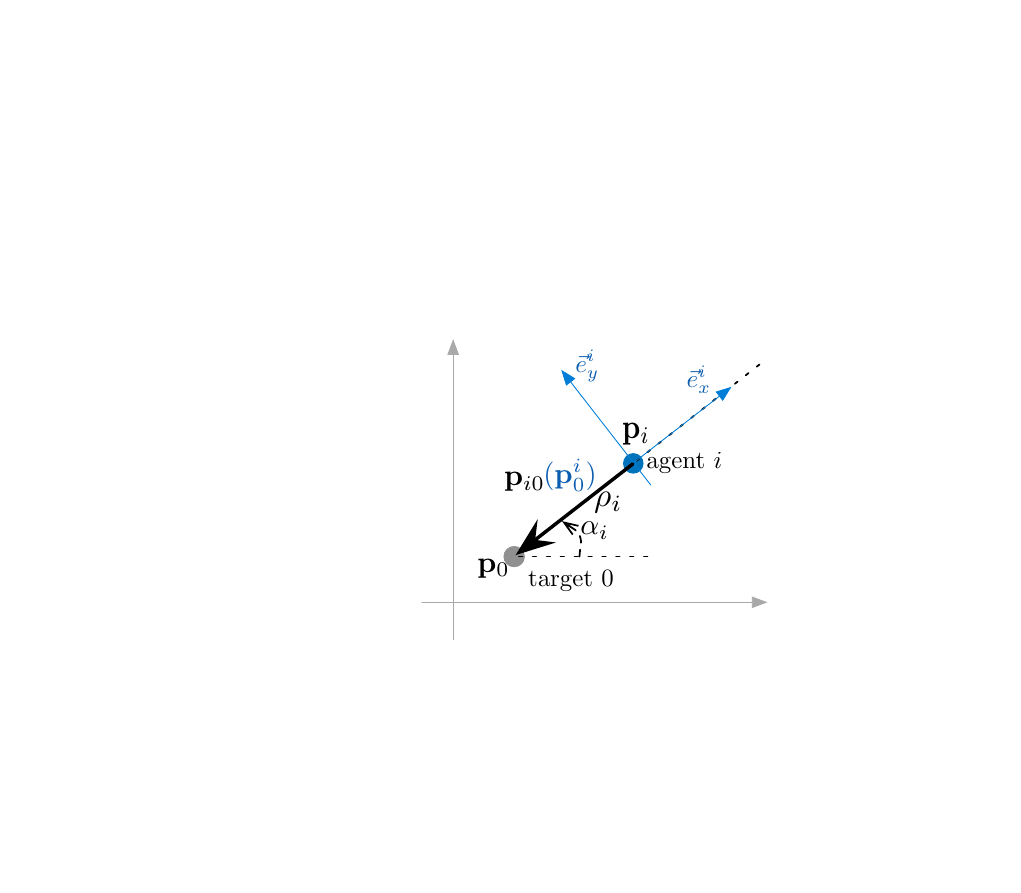}} \\
           \subfigure[Locally implementable control]{\label{fig:StatesRelationship}
           	\includegraphics[width=0.5\linewidth]{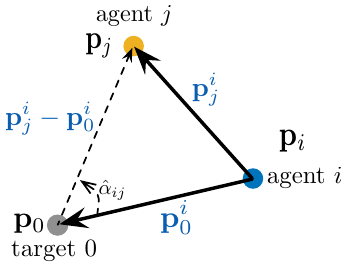}}
   \caption{Formation control in a plane.
        (a) Each agent $i$ can use only the local measurements in its \revision{local} frame.
        (b) The angular distance $\hat{\alpha}_{ij}$ between agent $i$ to its neighbor can be easily calculated by its local measurements.
   }\label{fig:Formation}
\end{center}   
\end{figure}
%

\subsection{Problem formulation}

Consider a group of $N$, $N \geq 2$, agents labeled $1$ to $N$ and a {static/}moving target labeled $0$ in a plane.
The $N$ agents and the target can move freely in the plane.
The $N$ agents' initial positions are NOT required to be distinguished from each other, whereas no agent occupies the same position as the target.
%
We associate the $N$-agent system with a directed graph $\mathbb{G}=(\mathcal{V},\mathcal{E})$,
where the node set $\mathcal{V}=\{1,2,\ldots,N\}$ represents $N$ agents,
and the edge set $\mathcal{E}=\{(j,i) : i,j \in \mathcal{V} \}$ characterizes local interactions between neighboring agents.
Thus a directed edge $(j,i) \in \mathcal{E}$ implies that agent $i$ can measure the relative information of agent $j$.
Then the neighbor set of agent $i$ is denoted as $\mathcal{N}_i = \{ j : (j,i) \in  \mathcal{E}  \}$.
Let matrix $A=(a_{ij})_{N\times N}$ be the adjacency matrix for $\mathbb{G}$,
and then $a_{ij}>0$ if and only if $(j,i)\in \mathcal{E}$, and $a_{ij}=0$ otherwise.
One knows that, for a directed graph, if every node, except a node called root, has exactly one parent, then it is called a directed tree; a spanning tree of a directed graph is a directed tree formed by the graph edges that connect all the nodes of the graph.
In this paper, the directed graph $\mathbb{G}(A)$ is required to contain a directed spanning tree.

Let $\mathbf{p}_i = [x_i, y_i]^T \in \mathbb{R}^2$ and $\mathbf{u}_i = [u_i^x, u_i^y]^T \in \mathbb{R}^2$
denote the position and control input to be designed of agent $i$, respectively. 
%
Each agent $i$ is described by a single-integrator dynamics model
\begin{eqnarray}\label{dynamic:global}
		\dot{\mathbf{p}}_i(t) = \mathbf{u}_i(t),  \qquad i \in \mathcal{V}.
\end{eqnarray}
%
The dynamic of the static/moving target is  described as follows  
\begin{eqnarray}\label{dynamic:global_0}
		\dot{\mathbf{p}}_0(t) = \mathbf{v}_0(t),
\end{eqnarray}
where $\mathbf{p}_0= [x_0, y_0]^T \in \mathbb{R}^2$ and  $\mathbf{v}_0 = [v_0^x, v_0^y]^T \in  \mathbb{R}^2$
denote the position and velocity of the target, respectively.
%


We emphasize that the above variables $\mathbf{p}_i$, $\mathbf{u}_i$, $\mathbf{p}_0$, and $\mathbf{v}_0$
are all described in a global coordinate frame.
However, a global coordinate frame or a common reference direction may not be allowed in real multi-robot systems.
%
Thus, for each agent $i$, we construct a moving frame, the \textbf{\revision{local} frame},
that is fixed on the agent with its origin at the representing point
and its $x$-axis opposite to the orientation of the ray extending from agent $i$ itself to the target.
The agent $i$'s \revision{local} frame is shown by $(\vec{e}^i_x, \vec{e}^i_y)$ in Fig. \ref{fig:Polar_relative}.
%

Let the superscript ${\cdot}^i$ denote the representing form of the corresponding vector in the \revision{local} frame of agent $i$.
Then the positions $\mathbf{p}_0$, $\mathbf{p}_j, j \in \mathcal{N}_i$ and the velocities  $\mathbf{u}_i$, $\mathbf{v}_0$ described in the global coordinate frame
can be converted to $\mathbf{p}_0^i$, $\mathbf{p}_j^i$, $\mathbf{u}_i^i$, $\mathbf{v}_0^i$  in the  \revision{local} frame of agent $i$
\begin{eqnarray} \nonumber
	\mathbf{p}_0^i(t) &=& \Phi_i(\alpha_i) [\mathbf{p}_0(t) - \mathbf{p}_i(t)]\\\nonumber	
	\mathbf{p}_j^i(t) &=& \Phi_i(\alpha_i) [\mathbf{p}_j(t) - \mathbf{p}_i(t)], \; j \in \mathcal{N}_i\\\nonumber
	\mathbf{u}_i^i(t) &=&  \Phi_i(\alpha_i) \mathbf{u}_i(t) \\\nonumber
	\mathbf{v}_0^i(t) &=& \Phi_i(\alpha_i) \mathbf{v}_0(t)
\end{eqnarray}
where  
\begin{eqnarray} \label{def:rotateM}
	 \Phi_i(\alpha_i)	 &=&  \left[\begin{array}{cc}  \cos{\alpha_i} & \sin{\alpha_i} \\  -\sin{\alpha_i} & \cos{\alpha_i}  \end{array}\right] ,
\end{eqnarray}
and $\alpha_i(t)$ is the angular of the ray extending from the target to agent $i$ at time $t$ in the global coordinate frame.


In this paper, the formation problem is formalized to design distributed controllers for each agent
by using only the local measurements of the target and its neighbors in each agent's own \revision{local} frame
such that all the agents asymptotically form a desired formation to keep the static/moving target as a reference point.
The desired formation is a general one \textbf{without} the requirement that all the desired distances between neighboring agents are equal \textbf{nor} the requirement that the desired distances between each agent and the target are equal.

To mathematically formulate the problem of interest, the following variables are introduced.
%
Let the variables  $\hat{\alpha}_{ij}(t)$ be the angular distance from agent $i$ to $j$ at time $t$.
The initial value $\hat{\alpha}_{ij}(0)$ is formed by counterclockwise rotating the ray extending from the target to agent $i$ until reaching agent $j$ at $t=0$,
thus $\hat{\alpha}_{ij}(0)\in[0,2\pi)$,
and the variables $\hat{\alpha}_{ij}(t)$ are required to be continuous.
%
Let $d_{ij} \in[0,2\pi)$ denote the desired angular spacing from agent $i$ to $j$, and $R_i\in \mathbb{R}$ denote the desired distance from agent $i$ to the target.
Then the $N$ agents' desired distribution pattern  is determined by 
$$d_{ij} \in[0,2\pi),  i \in \mathcal{V},  j \in \mathcal{N}_i$$ and  $$\mathbf{R} = [R_1,R_2,\ldots,R_N]^T \in \mathbb{R}^N.$$

In our problem setting,
%
each agent $i$ can only measure the relative positions $\mathbf{p}_0^i$, $\mathbf{p}_j^i,  j \in \mathcal{N}_i$,
and the velocity $\mathbf{v}_0^i$ of the target.
Furthermore, it is easy to check that,
agent $i$ can calculate the angular distance  $\hat{\alpha}_{ij},  j \in \mathcal{N}_i$ just using $\mathbf{p}_0^i$ and $\mathbf{p}_j^i$ \revision{based on} the definition of inner product.
%
\revision{That is, if the cross product $-\mathbf{p}_0^i \times (\mathbf{p}_j^i-\mathbf{p}_0^i ) \geq 0$, $\hat{\alpha}_{ij}$ happens to be the angle between two vectors, $-\mathbf{p}_0^i$ and $\mathbf{p}_j^i-\mathbf{p}_0^i$,
and if $-\mathbf{p}_0^i \times (\mathbf{p}_j^i-\mathbf{p}_0^i ) < 0$, $\hat{\alpha}_{ij}$ equals $2 \pi$ minus  the angle between $-\mathbf{p}_0^i$ and $\mathbf{p}_j^i-\mathbf{p}_0^i$
(see Fig. \ref{fig:StatesRelationship}).}

\par~

With the above preparation, we are ready to formulate the \emph{Formation Problem} of interest.

\par~

\begin{definition}[Admissible formation]\label{def:formationA}
	We say a prescribed \textbf{\em  formation} $(d_{ij}, \mathbf{R})$ is {\em admissible} if  \\
	\indent i) $\mathbf{R} = [R_1,R_2,\ldots,R_N]^T \in \mathbb{R}^N$ and $R_i>0$;
	
	ii) $d_{ij} \in[0,2\pi)$, $ i \in \mathcal{V},  j \in \mathcal{N}_i $,
	and there exists a vector
	\begin{equation}\nonumber
		\mathbf{d} = [d_1,d_2,\ldots,d_N]^T \in \mathbb{R}^N
	\end{equation}
	where $d_i \in[0,2\pi) $,
	such that $d_{ij}, \forall i \in \mathcal{V},  \forall j \in \mathcal{N}_i $ satisfies
	\begin{eqnarray}\nonumber
	    d_{ij} = \begin{cases}
	            d_j - d_i \; & \textrm{when} \;d_j - d_i \geq 0\\
	            d_j - d_i + 2\pi  \; & \textrm{when} \; d_j - d_i <0.
	            \end{cases}
	\end{eqnarray}\label{eq:neighbor}
\end{definition}

\par~

\begin{definition}[Formation Control in \revision{local} frame]\label{def:formationFS}
	Given an admissible formation characterized by $(d_{ij}, \mathbf{R})$ in a plane,
	design distributed  control laws for each agent $i$ using only  the local measurements in its \revision{local} frame, i.e.,
	\begin{eqnarray}\nonumber
		&\mathbf{u}_i^i(t) = \mathbf{u}_i^i(  \mathbf{p}_0^i(t), \mathbf{p}_j^i(t), \mathbf{v}_0^i(t), R_i, d_{ij} ), \\\nonumber
		&	 j \in \mathcal{N}_i, i \in \mathcal{V},
	\end{eqnarray}
	such that with any initial states $[\mathbf{p}_i^T(0),  \mathbf{p}_0^T(0)] \in\mathbb{R}^4$
	satisfying $\mathbf{p}_i^T(0) \neq \mathbf{p}_0^T(0), \forall i\in \mathcal{V}$,
	the solution to system(\ref{dynamic:global}) converges to some equilibrium point $\mathbf{p}^\ast$ satisfying	
	\begin{eqnarray}\nonumber
		\| {\mathbf{p}_0^i}^\ast \| &=& R_i \\ \label{goal:formation}
		\hat{\alpha}_{ij}^* &=& {d}_{ij}, \qquad i \in \mathcal{V},  j \in \mathcal{N}_i  .	 	
	\end{eqnarray}
\end{definition}

\par~

It is worth to emphasize that the formation $(d_{ij}, \mathbf{R})$ concerned in this paper can be any geometric pattern, so the controller to be designed is universal.
%
Especially, when there is no prescribed target, one can choose a proper agent to play the role of target, and then our proposed controller still works well.

\subsection{\keypoint{Preliminaries}}

Now we introduce the Barbalat's lemma and a useful result on the nonlinear consensus problem.

\par~

\begin{lemma}[Lemma 4.2 of \cite{SlLi91} ] \label{le:barbalat}
    If the differentiable function $f(t)$ has a finite limit as $t\rightarrow\infty$, and is such that $\ddot{f}$ exists and is bounded, then $\dot{f}(t)\rightarrow0$ as $t\rightarrow\infty$.
\end{lemma}

\par~

\begin{lemma}[Theorem 1 of \cite{LiChLu09}] \label{le:tanh}
    For a system modeled by 
    \begin{eqnarray}\label{lm2}
        \dot{x}_i(t)=\sum_{j=1}^n a_{ij}\phi_{ij}(x_j(t),x_i(t)),\quad i=1,2, \ldots, n,
    \end{eqnarray}
    where $x_i(t)\in \mathbb R$,
    $a_{ij}$ is the entry of matrix $A$, and     $a_{ij}\geq0$ for $i,j=1,2,\ldots,n$.
    Let $x = [x_1,x_2,\ldots,x_n]^T$ and  $y = [y_1,y_2,\ldots,y_n]^T$.
    If the following conditions for $\phi_{ij}$ hold

    i) $\phi_{ij}$ are continuous mappings and satisfy the local Lipschitz conditions;

    ii) $\phi_{ij}(x,y)=0 \iff x=y$;

    iii) $(x-y)\phi_{ij}(x,y)>0,\forall x\neq y$.

    \noindent then the system (\ref{lm2})  realizes consensus, i.e., $x_j(t)-x_i(t)\rightarrow0$ as $t\rightarrow\infty$, if and only if the directed graph $\mathbb{G}(A)$ has a spanning tree.
\end{lemma}
%


\par~

%
\section{Control law in \revision{local} frame}\label{se:continuous}
%

In this section, we propose a control law to solve the formation control problem, and then give theoretical analysis.

\subsection{Controller design}
%
%
%
%
%

From Definition \ref{def:formationFS}, the formation control problem can be divided into two sub-objectives which need to be concerned by each agent $i$.
The first sub-objective is to keep the desired distance to the static/moving target,
while the second one is to achieve the desired distances to its neighbors.
Thus we consider a controller in a decoupled form
\begin{eqnarray} \label{control:local}
	&\mathbf{u}_i^{i}(t) =
	     \lambda  \| \mathbf{p}_0^i  \|   f_i(t)
	        \left[ \begin{array}{c} \gamma (R_i^2 - \| \mathbf{p}_0^i \|^2) \\ \mu \end{array} \right]
	    + \mathbf{v}_0^i(t), \\\nonumber
	& i \in \mathcal{V},
\end{eqnarray}
where $\lambda>0, \gamma>0, \mu \neq 0$ are constants,
and $f_i:[0,\infty)\rightarrow  \in \mathbb{R}$ is a function to be designed to deal with the second sub-objective, 
while the rest part of the controller is mainly used to address the first sub-objective. 
Then we choose $f_i(t)$ as
\begin{eqnarray}\label{eq:u_f}
	f_i(t)  =  c + \mu \sum_{j\in \mathcal{N}_i} {a_{ij}}\tanh(\hat{\alpha}_{ij}-d_{ij})
\end{eqnarray}
%
where $c \in \mathbb{R}$ is a constant to be determined.
To ensure that controller (\ref{control:local}) combined with $f_i(t)$ in (\ref{eq:u_f}) still achieves the first sub-objective, 
a desired property of $f_i(t)$ is that $f_i(t) >0$ and $f_i(t)$ is bounded for all $t\geq0$, which will be discussed in the following subsections.
Such a property holds if we choose
$$c >|\mu|\max_{i\in\mathcal{V}}(\sum_{j\in\mathcal{N}_i}a_{ij}).$$


Now we have the complete form of the distributed controller $\mathbf{u}_i^{i}(t)$ in (\ref{control:local}) with  $f_i(t)$ in (\ref{eq:u_f}).

\subsection{Closed-loop dynamics of the $N$-agent system}
%

In order to analyze the equilibria of the $N$-agent system (\ref{dynamic:global}) under the proposed controller (\ref{control:local}) and (\ref{eq:u_f}),
consider the closed-loop dynamics of the $N$-agent system in the global coordinate frame.

For this purpose, we first introduce some variables in the global coordinate frame.
%
Let $\mathbf{p}_{i0}(t)= [x_{i0}, y_{i0}]^T$ be the relative position between agent $i$ and the target at time $t$,
\begin{eqnarray}\label{eq:p_i0}
	\mathbf{p}_{i0}(t) \triangleq \mathbf{p}_0(t) - \mathbf{p}_i(t)
	 = \Phi_i^{-1}(\alpha_i) \mathbf{p}_0^i(t) ,  \quad i \in \mathcal{V},
\end{eqnarray}
where $\Phi_i(\alpha_i)$ is given by (\ref{def:rotateM}).
Note that, from the definition of $\alpha_i(t)$ given in the previous section,
$\alpha_i(t)$ is the angular of the vector $-\mathbf{p}_{i0}(t)$ at time $t$ in the global coordinate frame.
Then the controller of agent $i$ can be represented in the global coordinate frame as
\begin{eqnarray} \label{control:global}
	\mathbf{u}_i(t) &=&  \Phi_i^{-1}(\alpha_i) \mathbf{u}_i^i(t) \\\nonumber
	&=& - \lambda f_i(t)
	    \begin{bmatrix}
	    	\gamma l_i(t) & -\mu \\
	    	\mu & \gamma l_i(t)
	    \end{bmatrix}
	\mathbf{p}_{i0} +\mathbf{v}_0(t), \\\nonumber
	&& i \in \mathcal{V},
\end{eqnarray}
where  
\begin{eqnarray} \label{cal:l_global}
	l_i(t) &=& R_i^2 - \| \mathbf{p}_{i0} \|^2
\end{eqnarray}
is the error between the current relative position and the desired one between agent $i$ and the target.

Substituting (\ref{control:global}) into the dynamic equations of the system (\ref{dynamic:global})  results in the closed-loop dynamics of the $N$-agent system in the global coordinate frame as
\begin{eqnarray} \label{system:global}
	& \dot{\mathbf{p}}_i(t) = - \lambda f_i(t)
	    \begin{bmatrix}
	    	\gamma l_i(t) & -\mu \\
	    	\mu & \gamma l_i(t)
	    \end{bmatrix}
	\mathbf{p}_{i0} +\mathbf{v}_0(t),  \\\nonumber
    & i \in \mathcal{V},
\end{eqnarray}
which can be rewritten equivalently using $\mathbf{p}_{i0}$ as
\begin{eqnarray} \label{system:global_relative}
	& \dot{\mathbf{p}}_{i0} = \lambda f_i(t)
	    \begin{bmatrix}
	    	\gamma l_i(t) & -\mu \\
	    	\mu & \gamma l_i(t)
	    \end{bmatrix}
	\mathbf{p}_{i0}, \qquad i \in \mathcal{V}.
\end{eqnarray}
%


Note that a limit-cycle oscillator shows up in the $N$-agent system's closed-loop dynamics (\ref{system:global_relative}).
For an oscillator having a stable limit cycle, it has the property that all trajectories in the vicinity of the limit cycle ultimately tend towards the limit cycle as time goes into infinity \cite{Kh06}.
For the closed-loop system (\ref{system:global_relative}), the limit cycle for each agent is a circle with the position of the target as the centroid and the desired distance from the agent to the target as the radius.
%

Inspired by the characteristics of the closed-loop dynamics, we represent the system (\ref{system:global_relative}) in the polar coordinate as
\begin{eqnarray}\label{system:rho}
	\dot{\rho}_i(t) &=& \lambda \gamma  \rho_i(t)(R_i^2-\rho_i^2(t))f_i(t), \\\label{system:alpha}
	\dot{\alpha}_i(t) &=& \lambda \mu f_i(t),
\end{eqnarray}
where $\rho_i(t) \triangleq \| \mathbf{p}_{i0}(t) \|$, $\alpha_i(t)$ is the angular of the vector $-\mathbf{p}_{i0}(t)$, and  
\begin{eqnarray}\nonumber
		\mathbf{p}_{i0}(t) = - \rho_i(t) \left[ \begin{array}{c}  \cos{\alpha_i(t)} \\  \sin{\alpha_i(t)} \end{array} \right].
\end{eqnarray}

We want to emphasize that the descriptions and variables in the global coordinate frame are only used for analysis purposes and are not known to the agents.

\subsection{\keypoint{Analysis of convergence}}

%
%
%

Now we are ready to analyze the convergence of the $N$-agent system in its polar coordinates form (\ref{system:rho}) and (\ref{system:alpha}).

\par~
\begin{lemma} \label{le:rho}
	For each agent $i$, 
	under the control law (\ref{control:local}),
	the solution to system (\ref{system:rho}) asymptotically converges to
	equilibrium point $\rho^*_i$ satisfying $\|\rho^*_i\|=R_i$
	if $f_i(t)>0$ {and $f_i(t)$ is bounded} for all $t \geq 0$.
\end{lemma}
%

%
\begin{proof}
	From equation (\ref{system:rho}), we can get the two equilibria of the system as $\rho_i=0$ and $\rho_i=R_i$, $i\in\mathcal V$.
	
	We first check the stability of the equilibrium point $\rho_i=0$.
	A Lyapunov function candidate is taken as
	\begin{eqnarray} \nonumber
		V_i(\rho_i)=\rho_i^2.
	\end{eqnarray}
	It's clear that $V_i(\rho_i)$ is positive definite and continuously differentiable.
	The derivative of $V_i(\rho_i)$ along the trajectories of the system is given by
	\begin{eqnarray}\nonumber
		\dot{V}_i(\rho_i)=2\rho_i \dot{\rho}_i=2\gamma\lambda(R_i^2-\rho_i^2)\rho_i^2 f_i(t).
	\end{eqnarray}
	In a small neighbourhood of $\rho_i=0$,
	$\dot{V}_i(\rho_i)$ is positive definite, since $\gamma>0, \lambda>0$ and $f_i(t)>0$. It turns out that $\rho_i=0$ is an unstable  equilibrium.
	
	To check the stability of the equilibrium point $\rho_i=R_i$,
	construct a Lyapunov function candidate as
	\begin{eqnarray} \nonumber
		W_i=(R_i^2-\rho_i^2)^2,
	\end{eqnarray}
	which is continuously differentiable.
	\revision{Its} derivative along the trajectories of the system can be calculated as
	\begin{align} \nonumber
		\dot{W}_i &= 4(R_i^2-\rho_i^2) \rho_i  \dot{\rho}_i  \\\nonumber
		&= -4\gamma \lambda f_i \rho_i^2 (R_i^2-\rho_i^2)^2.
	\end{align}	
	Since $f_i>0,\gamma>0,\lambda>0$, we have $\dot{W_i} \leq 0$, which implies $W_i(t)\leq W_i(0)$ as well. Then one knows that $\rho_i(t)$ is bounded because $$W_i(t)=[R_i^2-\rho_i^2(t)]^2\leq W_i(0).$$
	
	We further check the second derivative of $W_i$ as
	\begin{align} \nonumber
		\ddot{W}_i= &-8\gamma\revision{^2} \lambda\revision{^2} f_i^2 \rho_i^2 (R_i^2-\rho_i^2)^3 
		+16\gamma\revision{^2}  \lambda\revision{^2}  f_i^2 \rho_i^4 (R_i^2-\rho_i^2)^2  \\\nonumber
		&-4\gamma \lambda \rho_i^2 (R_i^2-\rho_i^2)^2 \dot{f}_i,
	\end{align}
	where the first and second terms on the right hand side are both bounded from above since $\rho_i(t)$ and $f_i(t)$ are bounded. The third term on the right hand side is also bounded because
	\begin{eqnarray} \nonumber
		\dot{f}_i=\mu \sum_{j\in \mathcal{N}_i}a_{ij}{\rm{sech}}^2(\hat{\alpha}_{ij}-d_{ij})(\dot{\alpha}_j-\dot{\alpha}_i)
	\end{eqnarray}
	is bounded. Then $\ddot{W}_i$ is bounded.
	From the Barbalat's Lemma (see Lemma \ref{le:barbalat}), we know $\lim_{t\rightarrow \infty}\dot{W_i}(t)=0$, from which we can get $\lim_{t\rightarrow \infty} \rho_i(t) =0$ or $\lim_{t\rightarrow \infty} \rho_i(t) = R_i$.
	Since $\rho_i=0$ is unstable, we known that every solution starting in \revision{$\rho_i(0) \in \mathbb{R} \setminus \{0\}$} converges to the equilibrium point $\rho_i(t) = R_i$ as $t\rightarrow\infty$.
	That completes the proof.
\end{proof}
%

\par~

Now we give the main result in this section.

\par~
\begin{theorem} \label{th:continuous}
	Suppose that the graph $\mathbb{G}(A)$ contains a directed spanning tree.
	Given an admissible formation characterised by $(d_{ij}, \mathbf{R})$,
	the formation control problem in \revision{local} frame is solved under the proposed controller (\ref{control:local}) with (\ref{eq:u_f}),
	if the parameter $c$ in the controller satisfies $c > |\mu|\max_{i\in\mathcal{V}}(\sum_{j\in\mathcal{N}_i}a_{ij})$.
\end{theorem}
%

%
\begin{proof}
    First of all, one can check that the designed $f_i(t)$ in (\ref{eq:u_f}) satisfies $f_i(t)>0$ and $f_i(t)$ is bounded if $c > |\mu|\max_{i\in\mathcal{V}}(\sum_{j\in\mathcal{N}_i}a_{ij})$.
    It follows that the condition of Lemma \ref{le:rho} is satisfied, so that the result of Lemma \ref{le:rho} still works here.

    Then, in order to prove this theorem, we just need to consider the other part of achieving the desired distances between neighbors.
    For this purpose, it suffices to show that  $\lim_{t\rightarrow\infty}\hat{\alpha}_{ij}(t)=d_{ij}$.

	Substituting (\ref{eq:u_f}) into  (\ref{system:alpha}), we get
	\begin{eqnarray}\nonumber
		\dot{\alpha}_i=\lambda \mu c+\lambda \mu^2 \sum_{j\in \mathcal{N}_i} a_{ij}\tanh(\hat{\alpha}_{ij} - d_{ij}).
	\end{eqnarray}
%
	Introduce variables $\xi_i(t) =\alpha_i-\lambda \mu c t-d_i$. Then we have
	\begin{eqnarray}\nonumber
		\dot{\xi}_i = \dot{\alpha}_i - \lambda \mu c
	\end{eqnarray}
	and 
	\begin{eqnarray}\nonumber
		\xi_j - \xi_i = \alpha_j - \alpha_i -d_{ij}=\hat{\alpha}_{ij}-d_{ij}.
	\end{eqnarray}
It should be noticed that the convergence of $\hat{\alpha}_{ij}-d_{ij}$ is equivalent to that of $\xi_j - \xi_i$.
	Then consider the system composed of $\xi_i$, which is given by
	\begin{align}\nonumber
		\dot{\xi}_i &= \lambda \mu^2 \sum_{j\in \mathcal{N}_i}a_{ij}\tanh(\xi_j-\xi_i) \\ \nonumber
		&= \lambda \mu^2 \sum_{j\in\mathcal{V}}a_{ij}\tanh(\xi_j-\xi_i).
	\end{align}
	Since $\tanh(\xi_i,\xi_j)$ satisfies the conditions in Lemma \ref{le:tanh} and $\mathbb{G}(A)$ contains a directed spanning tree,
	one can have that  $\lim_{t\rightarrow\infty}[\xi_j(t) - \xi_i(t)]=0$, i.e., $\lim_{t\rightarrow\infty}\hat{\alpha}_{ij}(t)=d_{ij}$.
\end{proof}
%

\par~


Furthermore, it is worth to emphasize that some parameters in our proposed controller (\ref{control:local}) show explicit physical meanings, which plays an important role in the motion characteristics of each agent $i$.
Particularly, taking (\ref{system:rho}) and (\ref{system:alpha}) into account,
it can be easily found that, at the stable equilibrium point ($\rho^\ast_i=R_i$),
the angular velocity relative to the target $\dot{\alpha}_i^\ast=\lambda \mu c$.
In other words, the parameters $\lambda>0, c>0, \mu \neq 0$ determines how fast the agent rotates around the target. 
Moreover, the sign of $\mu$ determines  which direction the agent rotate around the target and  $\mu >0$ (resp. $\mu <0$ ) corresponds to  counterclockwise (resp. clockwise) rotation.
In view of such a feature of these parameters, they can be selected more reasonable according to the request of the formation task and of the agents' motion restriction.


In the next section, we consider another practical issue arising when implementing the proposed control laws.

\par~

%
\section{Sampled-data based control law in \revision{local} frame}\label{se:sampled}
%

In practice,
robots are usually controlled in a discrete form and continuous-time control laws may not be directly implemented  to real robots, since there exist hardware constraints which may delay the execution time.
Hence, sampled-data based control laws are required. 
In this section, we investigate the convergence of the control laws proposed in the previous section for the case when sampled data approach is used.

\subsection{\keypoint{Sampled-data-based control law}}
%
%
Suppose that each agent samples synchronously and periodically with the same period and the  zero-order hold technique is used here.  Let $h$ be the sampling period. Then the sampled-data controller can be written as
\begin{align} \label{control:local_sd}
	\mathbf{u}_i^{i}(t) &= \lambda \| \mathbf{p}_0^i \|  \left[ \begin{array}{c} \gamma (R_i^2-\| \mathbf{p}_0^i(kh) \|^2) \\ \mu \end{array} \right] f_i(kh){+\mathbf{v}_0^i(kh)} ,\quad  \nonumber \\
	\forall t&\in[kh,kh+h),\quad i \in\mathcal{V},\quad k = 0,1,2,\cdots.
\end{align}%
%
Using this controller we can  get the closed-loop dynamics of the system in the global frame from (\ref{system:global_relative})
\begin{eqnarray}
	\leftline{$\mathbf{p}_{i0}(kh+h)=\mathbf{p}_{i0}(kh) {+ h[\mathbf{v}_0(t)-\mathbf{v}_0(kh)]} $}\nonumber
\end{eqnarray}%
\begin{eqnarray} \label{control:global_close_sd}
	+h\lambda
	\begin{bmatrix}
		 \gamma l_i(kh) & -\mu \\
		\mu &  \gamma l_i(kh)
	\end{bmatrix}
	\mathbf{p}_{i0}(kh) f_i(kh), \nonumber \\
	i \in\mathcal{V}, \quad k=0,1,2,\cdots,
\end{eqnarray}
where $f_i(kh)$ is given by
\begin{eqnarray} \label{eq:u_f_sd}
     f_i(kh) = c + \mu \sum_{j\in\mathcal{N}_i}a_{ij}\tanh\big(\hat{\alpha}_{ij}(kh)-d_{ij}\big).
\end{eqnarray}

To facilitate the analysis of the convergence, we focus on the case when the target is static,
i.e., $\mathbf{v}_0(t)=0,t>0$, and then rewrite the system in the polar coordinate as
\begin{align}
	\rho_i(kh+h)&=\rho_i(kh)  \nonumber \\
	&+h\gamma \lambda \rho_i(kh)\big(R_i^2-\rho_i^2(kh)\big)f_i(kh),\label{control:rho_sd}\\
	\alpha_i(kh+h)&=\alpha_i(kh)+h\lambda \mu c \nonumber \\
	&+h\lambda \mu^2 \sum_{j\in \mathcal{N}_i} a_{ij}\tanh(\hat{\alpha}_{ij}(kh)-d_{ij}). \label{control:alpha_sd}
\end{align}
Notice that the variable $\alpha_i(t)$ is not used in our control law but it is used to aid the analysis.

\subsection{\keypoint{Analysis of convergence}}
%
%
%
As stated before, the proposed control law can be divided into two parts. 
According to the decoupled design of the controller,
we will find the upper bound of the sampling period from these two aspects.

\par~

%
\begin{theorem} \label{th:sd}
    Suppose that the graph $\mathbb{G}(A)$ contains a directed spanning tree.
    Given an admissible formation characterised by $(d_{ij}, \mathbf{R})$ as well as a static target,
    the system (\ref{dynamic:global}) under the sampled-data controller (\ref{control:local_sd}) with (\ref{eq:u_f_sd}) has a locally exponentially stable equilibrium which is the desired formation,
    if
    the parameter $c$ in the controller satisfies $c > |\mu|\max_{i\in\mathcal{V}}(\sum_{j\in\mathcal{N}_i}a_{ij})$,
    and the sampling period $h$ satisfies
    \begin{eqnarray}
		0< h < h_{max}=\min\bigg(\frac{1}{2\gamma \lambda R^2 M},\frac{1}{\lambda \mu^2 d_{max}}\bigg),
	\end{eqnarray}
	where 
	 $$R=\max_{i\in \mathcal{V}}(R_i),$$
	$$d_{max}=\max_{i\in\mathcal{V}}(\sum_{j\in\mathcal{N}_i}a_{ij})\leq N,$$ 
	and $M = c +|\mu|\max_{i\in\mathcal{V}}(\sum_{j\in\mathcal{N}_i}a_{ij})$ is the upper bound of $f_i(t)$.
\end{theorem}
%

%
\begin{proof}
	We prove the local stability and   determine $h_{max}$ by considering (\ref{control:rho_sd}) and  (\ref{control:alpha_sd}).
	
	First, consider the dynamics of $\rho_i$ in (\ref{control:rho_sd}).
	Introduce error variables $\Delta_i(kh)=\rho_i(kh)-R_i$.
	Then we get
	\begin{align}\nonumber
		& \Delta_i(kh+h)=\Delta_i(kh)  \\\nonumber
		& -h\gamma \lambda f_i(kh)  \Delta_i(kh)\big(\Delta_i(kh)+R_i\big)\big(\Delta_i(kh)+2R_i\big)
	\end{align}
	which can be linearized around the equilibrium zero as
	\begin{align} \nonumber
		\Delta_i(kh+h)&=\Delta_i(kh)- 2 h\gamma \lambda f_i(kh)  R_i^2 \Delta_i(kh)  \\ \nonumber
		&=\Delta_i(kh)\Big(1-2h\gamma \lambda f_i(kh) R_i^2\Big).
	\end{align}
	It's clear that $\lim_{k\rightarrow\infty}\Delta_i(kh)=0$ if and only if
	\begin{equation}\label{eq:thm2_1}
	    \prod_{k=0}^\infty(1-2h\gamma \lambda f_i(kh) R_i^2)=0.
	\end{equation}
Since  	$f_i(kh)$ is lower bounded by a positive number, one has 
$$\sum_{k=0}^\infty f_i(kh)=\infty.$$
Therefore, if $1-2h\gamma \lambda f_i(kh) R_i^2\in(0,1), \forall k=0,1,\cdots$, then (\ref{eq:thm2_1}) holds.
	 Note that $f_i(kh)\leq M$. Then if
	 	\begin{eqnarray}
		h < \frac{1}{2R^2\gamma\lambda M},
	\end{eqnarray}
	we get $1-2h\gamma \lambda f_i(kh) R_i^2\in(0,1), \forall k=0,1,\cdots$.
	
	Next, consider the dynamics of $\alpha_i$ in (\ref{control:alpha_sd}).
	Using the similar transformation to that in the proof of Theorem \ref{th:continuous},
	we can also get the system of $\xi_i$ as
	\begin{align}\nonumber
		& \xi_i(kh+h) = \xi_i(kh)   \\\nonumber
		& + h\lambda \mu^2 \sum_{j\in\mathcal{V}}a_{ij}\tanh\big(\xi_j(kh)-\xi_i(kh)\big),
	\end{align}
	where $\xi_i(kh)=\alpha_i(kh)-kh\lambda \mu c-d_i$, $i \in \mathcal{V}$.
	Linearize the above system at the equlibrium $\xi_j- \xi_i = 0, \forall i\neq j$, i.e., $\hat{\alpha}_{ij}=d_{ij}, \forall i\neq j$, one can have
	\begin{align} \label{cal:xi_linear_sd}
		& \xi_i(kh+h) \\ \nonumber
		=& \xi_i(kh)(1-s \sum_{j\in\mathcal{V}}a_{ij}) + s \sum_{j\neq i}a_{ij}\xi_j(kh),
	\end{align}
	where $s\triangleq h \lambda \mu^2$ is constant.
	Furthermore, using the adjacent matrix $A$ and degree matrix $D$ of the directed graph $\mathbb{G}(A)$,
	one can rewrite (\ref{cal:xi_linear_sd}) in the matrix form
	\begin{eqnarray} \label{cal:xi_m}
		\xi(kh+h)=H\xi(kh),
	\end{eqnarray}
	where $H$ is a  matrix given by
	\begin{eqnarray}
		H=I-s(D-A),
	\end{eqnarray}
	and $\xi(kh)=\big(\xi_1(kh),\cdots,\xi_n(kh)\big)^T$.
	Note that the graph $\mathbb{G}(A)$ contains a directed spanning tree.
	Then one can check that the system (\ref{cal:xi_m}) reaches a consensus \cite{ReBe05},
	if the diagonal entries of $H$ are all positive.  
	It implies that  $\lim_{t\rightarrow \infty}(\xi_j-\xi_i)=0$, and hence $\lim_{t\rightarrow \infty}\hat{\alpha}_{ij}=d_{ij}$.
	
	To guarantee that the diagonal entries of $H$ are all positive, it is required that   
	$$1-h\lambda \mu^2 \max_{i\in\mathcal{V}}\bigg(\sum_{j\in\mathcal{V}} a_{ij} \bigg)>0,$$
	which implies 
	$$h < \frac{1}{\lambda \mu^2 \max_{i\in\mathcal{V}}\bigg(\sum_{j\in\mathcal{V}} a_{ij}\bigg)}.$$

	Now one can obtain an upper bound $h_{max}$ on the sampling period as 
	\begin{eqnarray}
		h_{max} = \min\bigg(\frac{1}{2\gamma \lambda R^2 M},\frac{1}{\lambda \mu^2 d_{max}}\bigg)
	\end{eqnarray}
	and if $h<h_{max}$, the desired formation defined in Definition \ref{def:formationFS} is locally exponentially stable.
\end{proof}
%

\par~

%
\section{Simulation results}\label{se:simulations}
%


In the simulations, we consider a system consisting of seven agents.
The target starts from the point $(0,0)$ in the plane without loss of generality.
The initial states of the seven agents are generated randomly.
In Fig. \ref{fig:simulation},
we show the simulation results of the \revision{continuous} controller (\ref{control:local}) and the sampled-data based controller (\ref{control:local_sd}), respectively.

\begin{figure}[thp]
    \centering
    \includegraphics[width=0.45\linewidth]{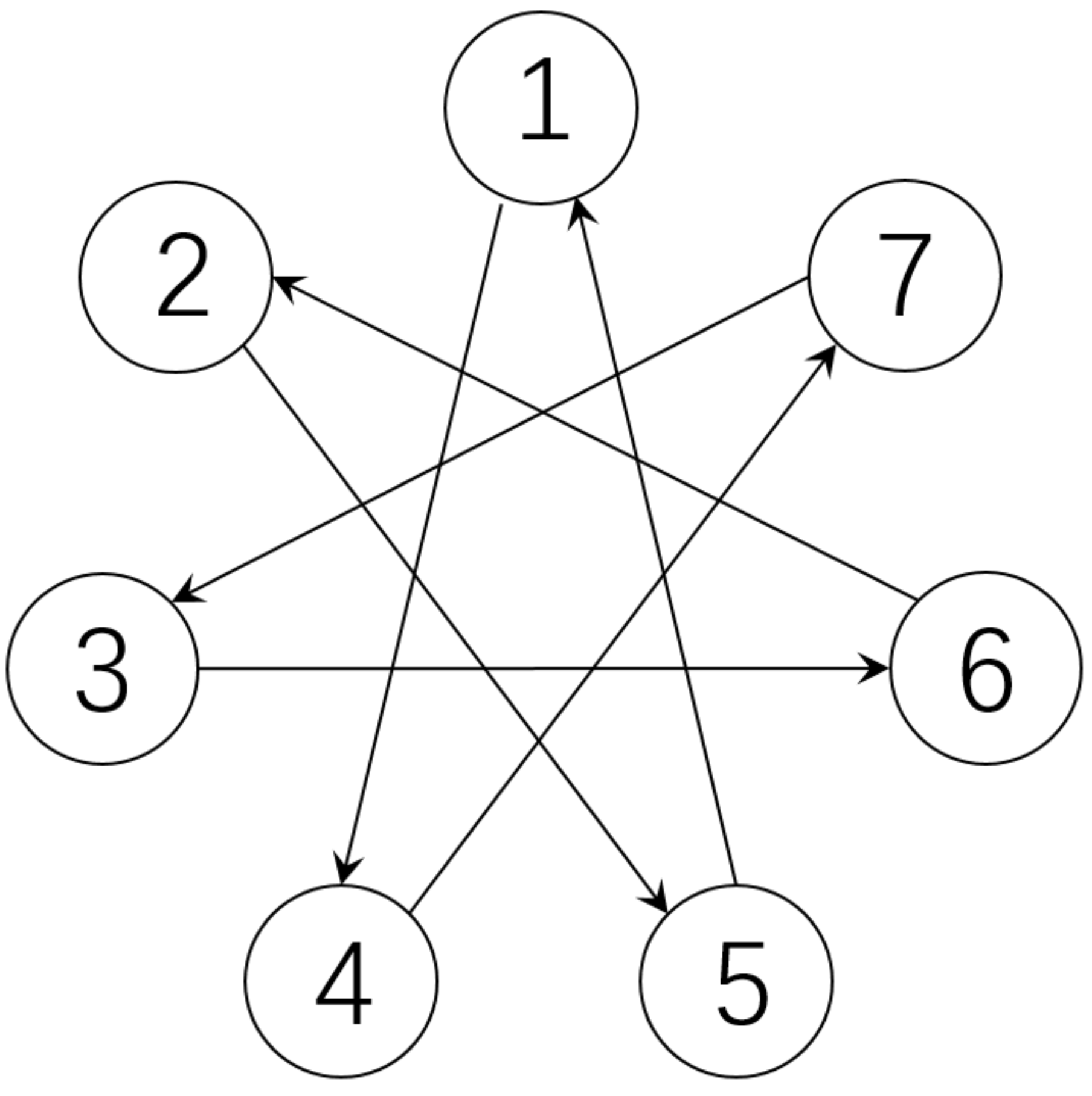}
    \caption{Neighbouring relationship of the agents.}
    \label{fig:topo}
\end{figure}

\begin{figure*}[thpb]
\begin{center}	
           \subfigure[Continuous-time case]{\label{fig:con_tra}
           \includegraphics[width=0.45\linewidth]{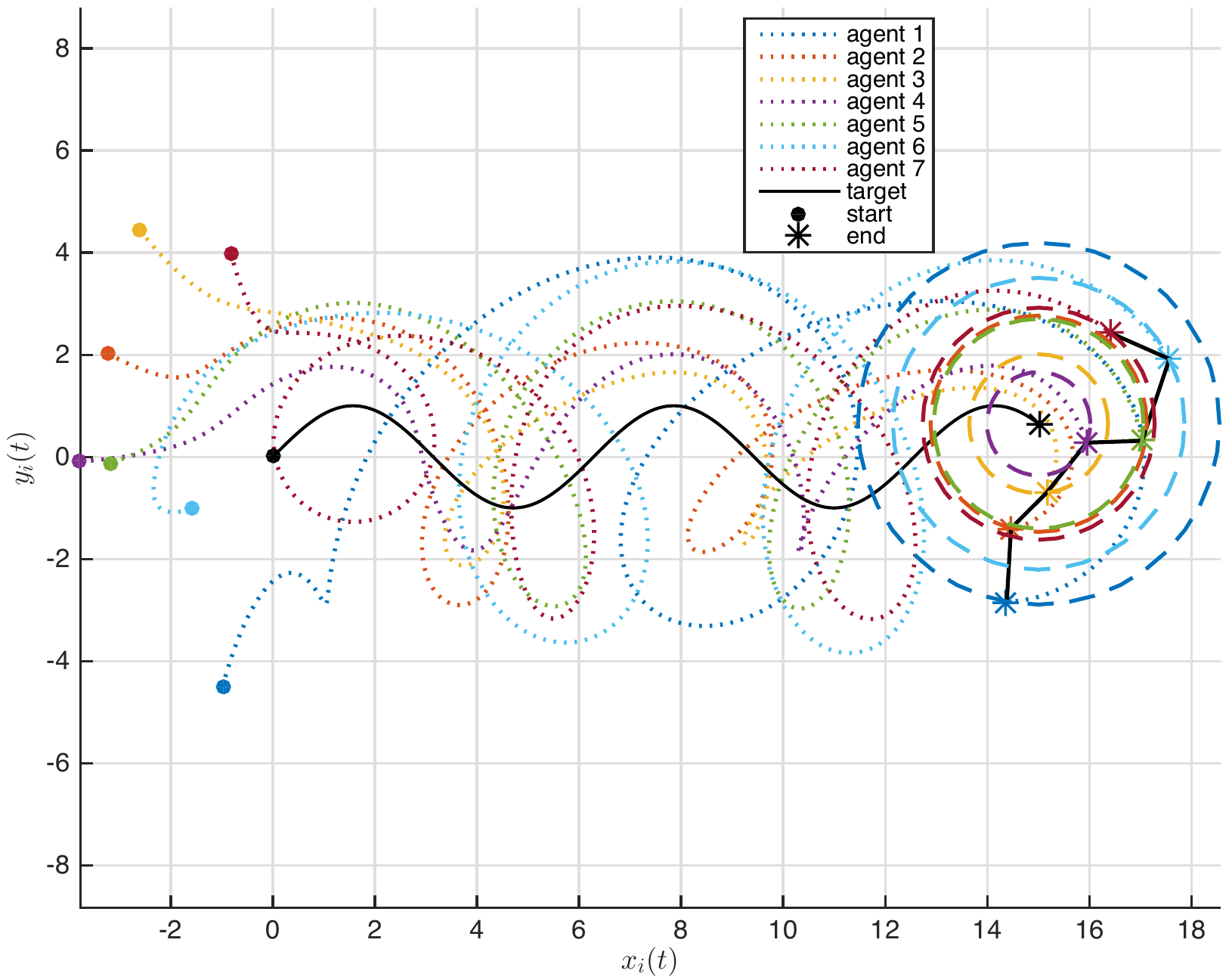}}
           \subfigure[Continuous-time case]{\label{fig:con_dis}
           \includegraphics[width=0.45\linewidth]{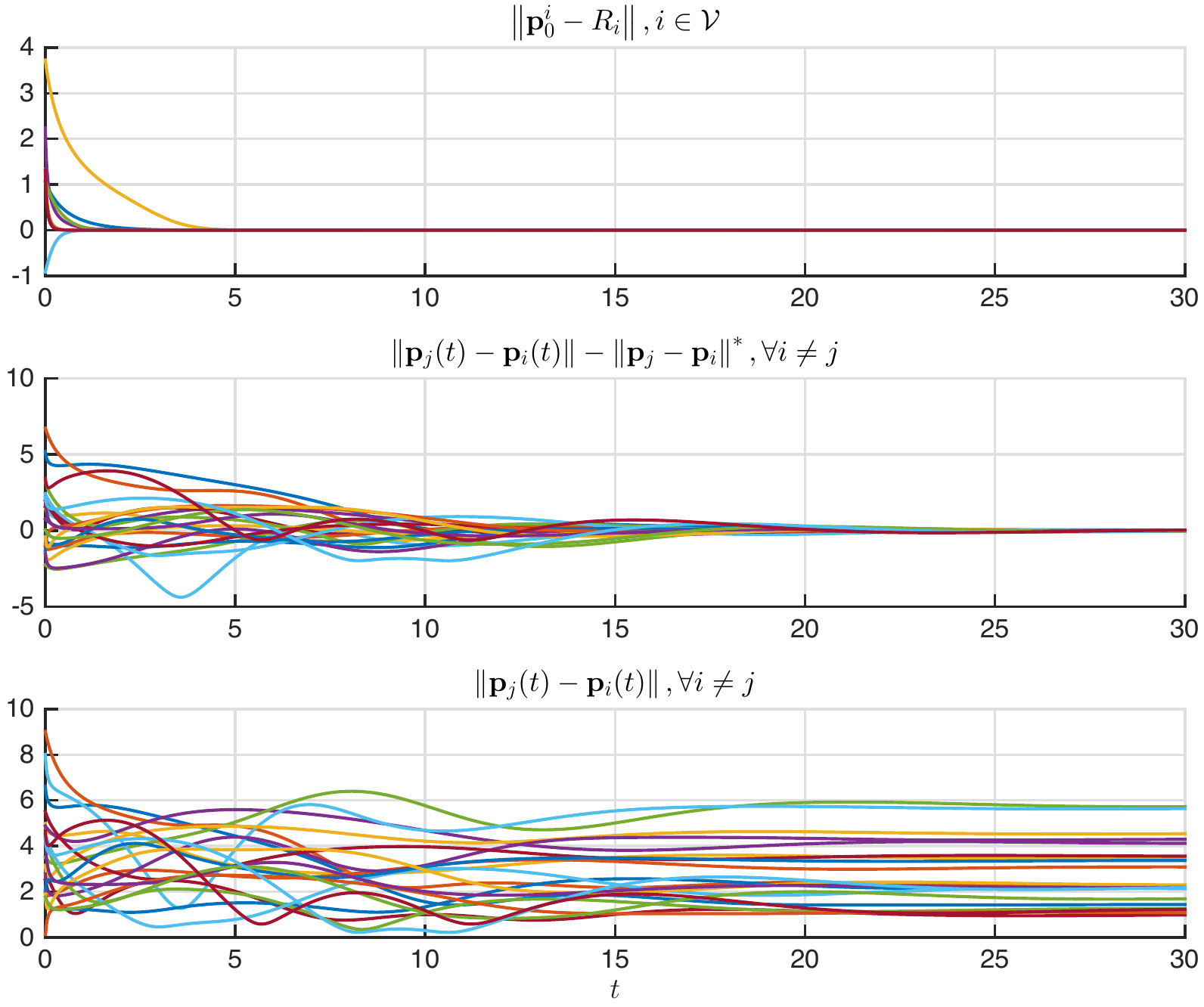}}
           \subfigure[Sampled-data-based case]{\label{fig:sap_tra}
           \includegraphics[width=0.45\linewidth]{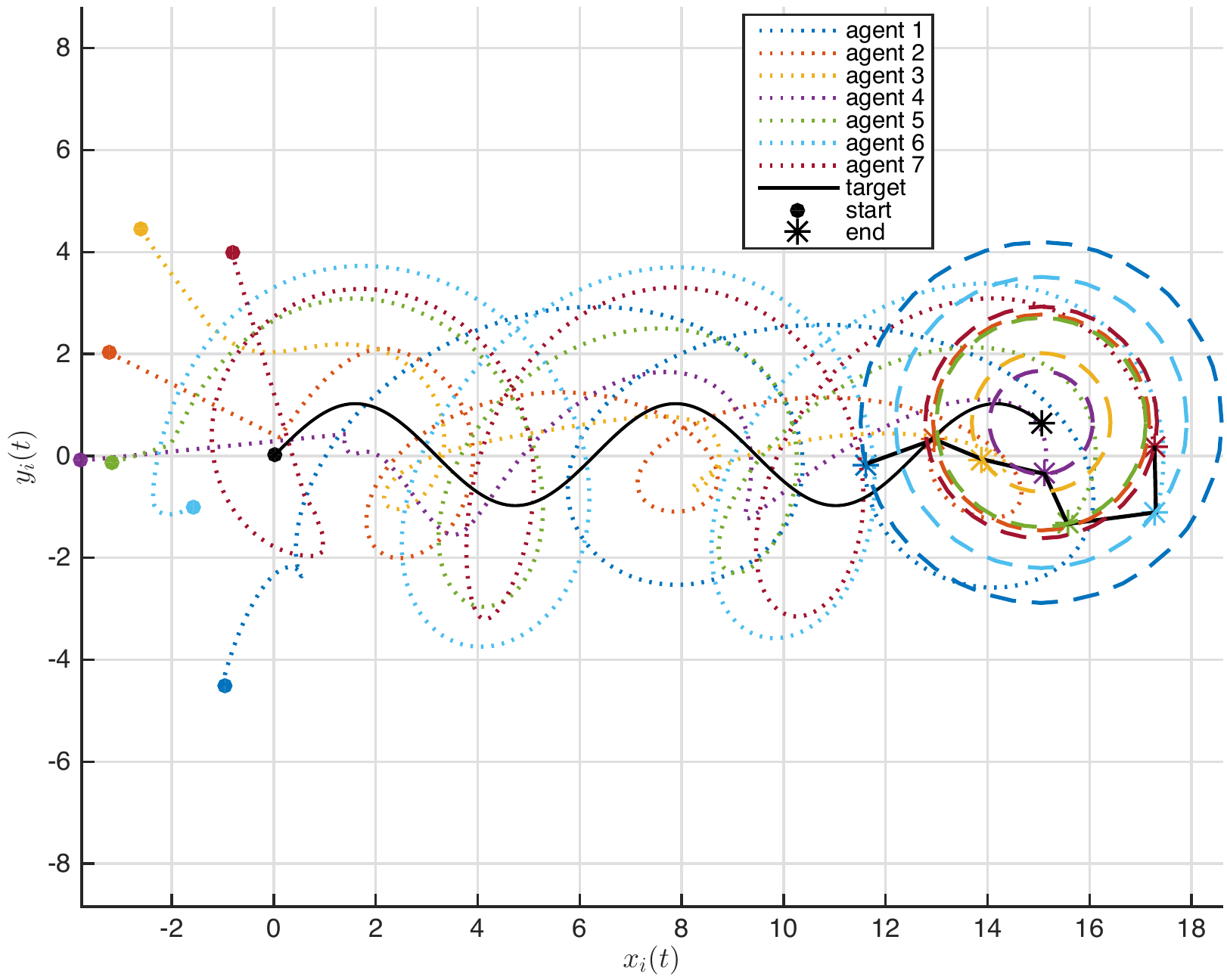}}
           \subfigure[Sampled-data-based case]{\label{fig:sap_dis}
           \includegraphics[width=0.45\linewidth]{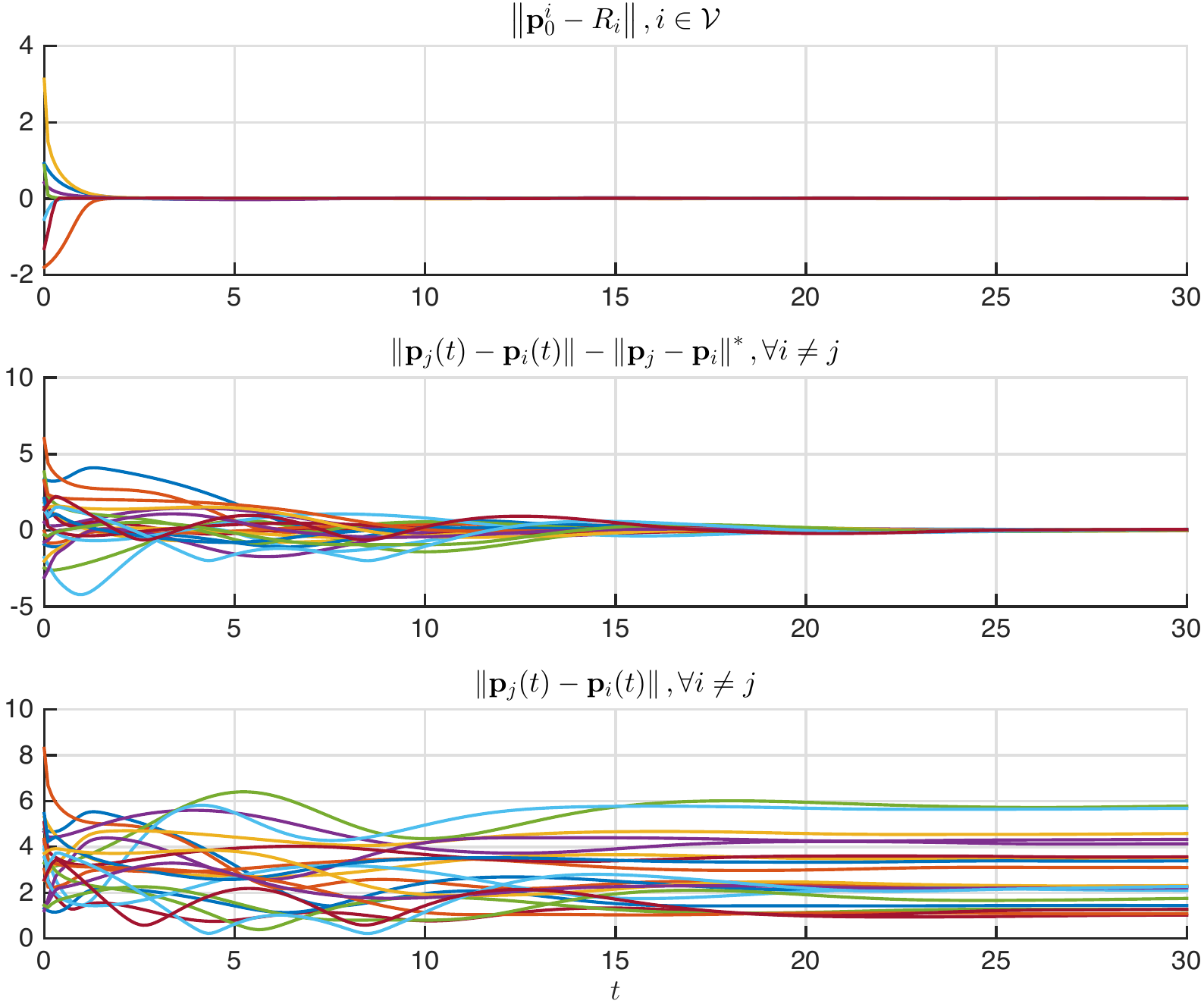}}
	\caption{Simulation results of the proposed controllers solving the formation control problem
			when $N=7$, and $c = 1.1 \mu = -1, \gamma = 1, \lambda = 0.5$.
		(a)(b) Continuous-time case. (c)(d) Sampled-data-based case with $h=0.1$.
		(a)(c) Trajectories of seven agents in the plane;
		(b)(d) Differences between each agent's current distance and the desired one to the target,
			differences between the current distances and the desired ones between all pairs of agents,
			and distances between all pairs of agents.
	}\label{fig:simulation}
\end{center}
\end{figure*}

For both scenarios, we choose the desired formation as the Big Dipper surrounding the target,
while the target is set to move along a sinusoidal curve.
The graph $\mathbb{G}(A)$, which describes the neighbouring relationship of the agents,
is set to be a directed loop shown in Fig \ref{fig:topo}. 
And $a_{ij}=1$ for  $(j,i)\in \mathcal{E}$ and $a_{ij}=0$ otherwise.
The parameters of the controllers are chosen as
$ c = 1.1, \mu = -1, \gamma = 1, \lambda = 0.5$.
Moreover, for the sampled-data system, from Theorem \ref{th:sd}, we have $h_{max} \approx0.0379.$
%
However, the upper bound of the sampling period calculated according to Theorem~\ref{th:sd}  is rather conservative. The controller still works when $h$ takes a larger value.
Thus we choose the sampling period $h=0.1$ to perform the simulation under the sampled-data controller.


The simulation results clearly indicate that these agents asymptotically converge to the prescribed formation under the proposed controllers (\ref{control:local}) and (\ref{control:local_sd}) for the \revision{continuous} case and the sampled-data based case, respectively.
Especially, it is shown that the sampled-data based controller still works when the target is moving, although we only give the theoretical analysis for the situation of the static target.
Moreover, since the distances between any two agents are all positive, the collision avoidance is guaranteed, which makes the controller more suitable to apply to real robots.





\par~

%
\section{Conclusions}\label{se:conclusions}
%

In this paper, we have studied the formation control problem for a group of mobile agents which can only measure the local information in their own \revision{local} frame. 
The problem includes two sub-objectives of forming a desired geometric pattern and keeping the formed geometric pattern with a desired distance to a static/moving target.
Then using the idea of decoupled design, we have designed a distributed local controller combining two parts to solve the control problem.
Furthermore, the sampled-data based controller has been proposed.
Theoretical analysis has been provided to show the convergence of the system under our proposed controller,
for both the \revision{continuous} case with a static/moving target and the sampled-data one with a static target.
Finally, numerical simulations have been performed to demonstrate the effectiveness and performance of the controllers.
%

Notice that one of the nice properties of our proposed controllers is that no collision between agents ever takes place, which makes the controller more suitable to apply to real robots. Such a property has been shown clearly via simulations, however, the theoretical analysis on this property is missing and is under investigation.

\par~

%
%

%
%
%
%
%
%

%
%
%
%
\bibliographystyle{IEEEtran}        


%
%

\end{document}